\documentclass[11pt]{article}
\usepackage{amssymb}
\usepackage{amsthm}
\usepackage{amsmath}
\def \C{\hbox{\sf\rlap{\kern.25em\vrule width.1em height1.4ex depth-.06ex}C}}
\def \R{{\sf I\kern-1.5ptR}}

\def \qed{\hspace*{2mm} \hfill $\Box $\bigskip}

\newtheorem{theorem}{Theorem}[section]
\newtheorem{proposition}[theorem]{Proposition}
\newtheorem{lemma}[theorem]{Lemma}

\newtheorem{corollary}[theorem]{Corollary}

\theoremstyle{remark}
\newtheorem{example}[theorem]{Example}
\newtheorem{remark}[theorem]{Remark}

\makeatletter
\renewenvironment{proof}[1][\proofname]{\par
\pushQED{\qed}%
\normalfont \topsep6\p@\@plus6\p@\relax
\trivlist
\item\relax
{\itshape
#1\@addpunct{.}}\hspace\labelsep\ignorespaces
}{%
\popQED\endtrivlist\@endpefalse
}
\makeatother
\numberwithin{equation}{section}
\begin{document}
\thispagestyle{plain}
\begin{center}
  \Large
    Bergman orthogonal polynomials and the Grunsky matrix
  \\[20pt]
  \normalsize by \\[15pt]
  \large Bernhard Beckermann\footnote{Laboratoire Painlev\'e UMR 8524 (AN-EDP), UFR Math\'ematiques --
M3, Univ.\ Lille, F-59655 Villeneuve d'Ascq CEDEX, France. E-mail:
{\tt bbecker@math.univ-lille1.fr}. Supported  in  part  by  the  Labex  CEMPI
   (ANR-11-LABX-0007-01).} and Nikos Stylianopoulos\footnote{Department of Mathematics and Statistics, University of Cyprus, P.O. Box 20537, 1678 Nicosia,
Cyprus.
E-mail: {\tt nikos@ucy.ac.cy}. Supported  in  part  by  the  University of Cyprus grant 3/311-21027.}
\end{center}
\begin{abstract}
    By exploiting a link between Bergman orthogonal polynomials and the Grunsky matrix, probably first observed by K\"uhnau in 1985, we improve on some recent results on strong asymptotics of Bergman polynomials outside the domain $G$ of orthogonality, and on the entries of the Bergman shift operator.
    In our proofs we suggest a new matrix approach involving the Grunsky matrix, and use well-established results in the literature relating properties of the Grunsky matrix to the regularity of the boundary of $G$ and the associated conformal maps. For quasiconformal boundaries this approach allows for new insights for Bergman polynomials.
\end{abstract}

\hfill 

\def\fa{f}
\def\fe{r}
\def\fH{\widetilde H}

\section{Introduction and main results}

Let $G$  be a bounded simply connected domain in the complex plane $\mathbb C$,
with boundary $\Gamma$. We define, for $n \geq 0$, the so-called Bergman orthogonal polynomials
$$
     p_n(z)=\lambda_n z^n + \mbox{~terms of smaller degree} , \quad \lambda_n > 0,
$$
satisfying
$$
    \langle p_n,p_m \rangle_{L^2(G)} := \int_G p_n(z) \overline{p_m(z)} \, dA(z) = \delta_{m,n} ,
$$
where $dA(z)=dxdy$ denotes planar Lebesgue measure.
There is a well developed theory regarding Bergman polynomials. For their basic
properties, and the asymptotic behavior including that of their zeros see  \cite{Carleman1923,Suetin1974,Gaier1987,Saff1990,Gust2008,Gust2009,DrMD,Nikos2013,Nikos2014,Nikos2015} and the references therein.
In describing these we need two conformal maps.

Denote by $\mathbb D$ the open unit disk, by $\mathbb D^*$ the exterior of the closed unit disk, by $\Omega$ the exterior of the closure $\overline{G}$ of $G$, and consider the Riemann map from $\Omega$  onto $\mathbb D^*$
$$
    w=\phi(z)=\phi_{1}z + \phi_0 + \phi_{-1} z^{-1}+...,
$$
with inverse map $$
   z=\psi(w)=\phi^{-1}(w)=\psi_{1}w + \psi_0 + \psi_{-1} w^{-1}+...,
$$
where $\gamma:=\phi_{1}=1/\psi_{1}=1/\psi'(\infty)>0$ is the inverse of the logarithmic capacity of $G$.

We will make also use of the $n$-th Faber polynomial $F_n$. This is the polynomial part of the Laurent expansion at $\infty$ of $\phi(z)^n$, and hence $F_n$ is of degree $n$, with leading coefficient $\gamma^{n}$. The corresponding Grunsky coefficients $b_{\ell,n}=b_{n,\ell}$ are defined by the generating series
\begin{equation} \label{Grunsky_coefficients}
   \log \Bigl( \frac{\psi(w)-\psi(v)}{\psi'(\infty)(w-v)}\Bigr)
   = - \sum_{n=1}^\infty \sum_{\ell=1}^\infty b_{n,\ell} w^{-\ell} v^{-n}
   = - \sum_{n=1}^\infty \frac{1}{n} \Bigl( F_n(\psi(w))-w^n \Bigr) v^{-n},
\end{equation}
which is analytic and absolutely convergent for $|w|>1,|v|>1$, see, e.g., \cite[Chapter 3, Eqns (8) and (10)]{Pommerenke1975}.

In what follows it will be convenient to work with the normalized Grunsky coefficients
\begin{equation} \label{normalized_Grunsky_coefficients}
      C_{n,k} = C_{k,n} = \sqrt{n+1}\, \sqrt{k+1}\, {b_{n+1,k+1}},
      \quad \mbox{for $n,k=0,1,2,...$,}
\end{equation}
for which the Grunsky inequality \cite[Theorem~3.1]{Pommerenke1975} reads as follows: For any integer $m \geq 0$ and any complex numbers $y_0,y_1,...,y_m$ there holds
\begin{equation} \label{Grunsky_inequality}
       \sum_{n=0}^\infty \Bigl| \sum_{k=0}^m C_{n,k} y_k \Bigr|^2
       \leq  \sum_{k=0}^m | y_k |^2 .
\end{equation}
We note that \eqref{Grunsky_inequality} holds under the sole assumption that $\overline G$ is a continuum, i.e., closed and connected, and $\Omega$ is the component of $\overline{\mathbb C} \setminus \overline{G}$ that contains $\infty$.

Under various assumptions on $\Gamma$ many results are scattered over the literature establishing that the Bergman polynomial $p_n$ behaves like a suitably normalized derivative of the Faber polynomial $F_{n+1}$, namely like
\begin{equation} \label{faber}
    \fa_n(z):=\frac{F'_{n+1}(z)}{\sqrt{\pi}\sqrt{n+1}} = \sqrt{\frac{n+1}{\pi}} \gamma^{n+1} z^n + \mbox{~terms of smaller degree}.
\end{equation}

Taking derivatives with respect to $w$ in \eqref{Grunsky_coefficients}, using
\eqref{normalized_Grunsky_coefficients} and comparing like powers of $v$ leads, for $|w|>1$, to the formula
\begin{equation} \label{faber_residual}
    \fe_n(w):=\psi'(w)\fa_n(\psi(w)) - \sqrt{\frac{n+1}{\pi}} w^n
    = - \sum_{\ell=0}^\infty \sqrt{\frac{\ell+1}{\pi}}w^{-(\ell+2)} C_{\ell,n}.
\end{equation}
The Grunsky inequality \eqref{Grunsky_inequality} allows us to conclude\footnote{We do not need further assumptions on the boundary like $\Gamma$ being rectifiable, compare with \cite[Lemma~2.1]{Nikos2013}.}
that
$r_n\in \mathbb L^2(\mathbb D^*)$,
and more precisely
\begin{equation} \label{norm_residual}
    \varepsilon_n := \| \fe_n \|_{L^2(\mathbb D^*)}^2 = \sum_{j=0}^\infty |C_{j,n}|^2 \leq 1.
\end{equation}
The upper bound in (\ref{norm_residual}) follows by choosing $y_0=y_1=...=y_{n-1}=0$ and
$y_n=1$ in (\ref{Grunsky_inequality}). For the equality connecting the $L^2$-norm with the summation over
Grunsky coefficients see the last part of the proof of Lemma~\ref{M=I-C*C} below.
In the same lemma we show that the relation
\begin{equation}\label{eq:eps-fn}
\varepsilon_n = 1 -\| \fa_n \|^2_{L^2(G)}
\end{equation}
holds under the sole assumption that the boundary $\Gamma$ of $G$ has zero area.
For comparison we note that, although there is a different normalization for $\fa_n$, the quantity $\varepsilon_n$ in \eqref{norm_residual} coincides with that of \cite[Eqn.\ (2.21)]{Nikos2013}.


Most of the results of this paper require the additional assumption that the boundary $\Gamma$ is quasiconformal. We refer the reader to \cite[\S 9.4]{Pommerenke1975} for a definition and elementary properties, and recall that a piecewise smooth (in particular piecewise analytic) Jordan curve $\Gamma$ is quasiconformal if and only if it
has no cusps. Also, a quasiconformal curve is a Jordan curve, with two-dimensional Lebesgue measure zero, see, e.g., \cite[Section B.2.1]{AndrBlatt}. In our considerations in \S 2 the following fact will be important \cite[Theorem~9.12 and Theorem~9.13]{Pommerenke1975}: {\it we can improve the Grunsky inequality \eqref{Grunsky_inequality} exactly for the class of quasiconformal curves, by inserting a factor strictly less than one on the right-hand side of \eqref{Grunsky_inequality}.}

For quasiconformal $\Gamma$, the following double inequality have been established by the second author in \cite[Lemma 2.4 and Theorem 2.1]{Nikos2013}:
\begin{equation} \label{asymp_leading}
        \varepsilon_n \leq 1 - \frac{n+1}{\pi} \frac{\gamma^{2n+2}}{\lambda_n^2}
        \le c(\Gamma)\varepsilon_n,
\end{equation}
where $c(\Gamma)$ is a positive constant that depends on $\Gamma$ only.
In other words, the leading coefficient $\lambda_n$ of $p_n$ behaves like that of $f_n$ if and only if $\varepsilon_n\to 0$, $n\to\infty$.

By using (\ref{eq:eps-fn}), \eqref{asymp_leading}  and Pythagoras' theorem, it is not difficult to obtain the estimate
\begin{equation} \label{asymp_norm}
        \| \fa_n - p_n\|_{L^2(G)}^2
        = \mathcal O(\varepsilon_n),
\end{equation}
see, e.g., \cite[p.~71]{Nikos2013}.
This means that if $\varepsilon_n \to 0$, then we also get $L^2$ asymptotics on the support of orthogonality for $p_n$. Furthermore, by employing the arguments in \cite[p.\ 2449]{Nikos2015b}, the relation \eqref{asymp_norm} leads to the relative asymptotics
\begin{equation} \label{asymp_outside}
        \frac{\fa_n(z)}{p_n(z)} = 1 + \mathcal O(\sqrt{\varepsilon_n}) ,
\end{equation}
for $z$ outside the closed convex hull $\textup{Co}(G)$ of $G$.
Note that, by Fejer's theorem \cite[Theorem~2.1]{Gaier1987}, $p_n$ is zero free outside $\textup{Co}(G)$.

In Remark~\ref{rem_equivalence} below we provide alternate (and shorter) proofs of \eqref{asymp_leading} and \eqref{asymp_norm} based on a novel matrix approach. Moreover, a sharper result related to \eqref{asymp_outside} will be established in Theorem~\ref{thm1}.

Here we review a number of important cases where the behavior of $\varepsilon_n$ is known. Further details are given in \S3 below.

When $\Gamma$ is analytic, then the function $\psi$ has an analytic and
univalent continuation to $|w|>\rho$, for some $\rho\in [0,1)$. We write $\Gamma\in U(\rho)$ for the smallest such $\rho$. The asymptotic behavior of Bergman polynomials for $\Gamma\in U(\rho)$ has been first derived by Carleman in \cite{Carleman1923} where, in particular, it was shown that
\begin{equation} \label{in_carleman}
      1 - \frac{n+1}{\pi} \frac{\gamma^{2n+2}}{\lambda_n^2}
        = \mathcal O(\rho^{2n}).
\end{equation}
For $\Gamma$ in the same class K\"uhnau noted in \cite[Eqn.\ (9)]{Kuehnau1986}  the inequality $|C_{\ell,n}|\leq \rho^{\ell+n+2}$ and concluded using \eqref{norm_residual} that $\varepsilon_n=\mathcal O(\rho^{2n})$. Thus \eqref{in_carleman} can be also deduced from \eqref{asymp_leading} by means of the Grunsky coefficients. The particular case of an ellipse discussed in \S\ref{example_ellipse} below shows that \eqref{in_carleman} is not sharp, because in this case 
$\varepsilon_n=\rho^{4n+4}$. For a recent discussion of asymptotics inside $G$ when $\Gamma\in U(\rho)$ we refer to Mi\~{n}a-D{\'\i}az \cite{MD2008}.

If there is a parametrization of $\Gamma$ having a derivative of order $p\geq 1$ which is H\"older continuous with index $\alpha\in (0,1)$, then we write $\Gamma\in\mathcal C(p,\alpha)$.
Under the assumption $\Gamma\in \mathcal C(p+1,\alpha)$, $p \geq 0$, Suetin has shown in \cite[Lemma~1.5]{Suetin1974} that $\varepsilon_n=\mathcal O(1/n^{\beta})$, with $\beta=2p+2\alpha$. When $\beta>1$, the upper bound of \eqref{asymp_leading} can be found in \cite[Theorem 1.1]{Suetin1974}.

Finally, the case of piecewise analytic $\Gamma$ without cusps has been studied recently by the second author,
who obtained in \cite[Theorem 2.4]{Nikos2013} the estimate
$\varepsilon_n=\mathcal O(1/n)$. An example of two overlapping disks considered by Mi\~{n}a-D{\'\i}az in \cite{MinaDiaz} shows that this estimate cannot be improved, since $\liminf_n n\varepsilon_n >0$, for this particular $\Gamma$.

To complete the picture we note that, as Proposition~\ref{co2} below shows, the estimate  $\varepsilon_n=\mathcal{O}({1}/{n})$ \emph{cannot hold for all quasiconformal curves} $\Gamma$.

For easy reference we summarize:
\begin{equation} \label{summarize}
   \varepsilon_n = \left\{\begin{array}{ll}
   \mathcal O(\rho^{2n}), & \mbox{if $\Gamma\in U(\rho)$,}
   \\
   \mathcal O({1}/{n^{2(p+\alpha)}}), & \mbox{if $\Gamma\in \mathcal C(p+1,\alpha)$,}
   \\
   \mathcal O({1}/{n}), & \mbox{if $\Gamma$ is piecewise analytic without cusps.}
\end{array}\right.
\end{equation}
Motivated by the three preceding cases we will assume hereafter that
\begin{equation} \label{assumption_epsilon}
     \varepsilon_n=\mathcal O(1/n^{\beta}), \quad \mbox{for some $\beta>0$.}
\end{equation}

From the reproducing kernel property in $L^2(\mathbb D^*)$ we have that, for $|w|>1$,
$$
    \fe_n(w)  =   \int_{\mathbb D^*} \fe_n(u) K(u,w)dA(u),
    \quad \mbox{where}
    \quad K(u,w)=\frac{1}{\pi (w\overline{u}-1)^2}.
 $$
Then, using the Cauchy-Schwarz inequality, the definition of $\varepsilon_n$ in \eqref{norm_residual} and the fact that
$K(w,w)=\| K(\cdot,w) \|_{L^2(\mathbb D^*)}^2$, we conclude
\begin{equation} \label{Faber_outside_bis}
     | \fe_n(w) | \leq \| r_n \|_{L^2(\mathbb D^*)} \,
        \| K(\cdot,w) \|_{L^2(\mathbb D^*)}
     = \sqrt{\frac{\varepsilon_n}{\pi}}\frac{1}{|w|^2-1}.
\end{equation}
Therefore, it follows from \eqref{faber_residual}, for $w=\phi(z)$, that
\begin{equation} \label{Faber_outside}
    \sqrt{\frac{\pi}{n+1}} \frac{\fa_n(z)}{\phi'(z)\phi(z)^n} - 1
    = \sqrt{\frac{\pi}{n+1}} \frac{\fe_n(w)}{w^n}
\end{equation}
tends to zero uniformly on compact subsets of $\Omega$ with a geometric rate.

Our first result describes strong asymptotics for Bergman polynomials outside the support of orthogonality.
\begin{theorem}\label{thm1}
   Let $\Gamma$ be quasiconformal, and assume that $\varepsilon_n=\mathcal O(1/n^\beta)$, for some $\beta>0$. Then
   \begin{equation} \label{eqn.thm1}
    \sqrt{\frac{\pi}{n+1}} \frac{p_n(z)}{\phi'(z)\phi(z)^n} - 1
         = \mathcal O(\frac{\sqrt{\varepsilon_n}}{n^{\beta/2}}), \quad n \to \infty,
   \end{equation}
   uniformly on compact subsets of $\Omega$.
\end{theorem}
The result of Theorem~\ref{thm1} should be compared with \cite[Theorem~1.2]{Nikos2013}, for piecewise analytic $\Gamma$ without cusps ($\beta=1$) and with \cite[Theorem~1.4]{Suetin1974}, for sufficiently differentiable $\Gamma$ where, in both cases, the rate obtained is $\mathcal O(\sqrt{\varepsilon_n})$, the same as in \eqref{asymp_outside}.
When $\Gamma\in U(\rho)$, then we can still apply Theorem~\ref{thm1}, with $\beta$ any positive constant, and compare it
with the corresponding result of Carleman \cite[Theorem~2.2]{Gaier1987}, which gives the rate
$\sqrt{n\varepsilon_n}=\mathcal O(\sqrt{n}\rho^n)$, cf. \cite[p.~1983]{DrMD} and
\cite[Theorem 1]{MD2008}.
That is, in all these three cases, \eqref{eqn.thm1} yields an improvement of order $\mathcal O(1/n^{\beta/2})$ in the estimated rate of convergence. We should point out, however, that in contrast to the cited works, where
(\ref{eqn.thm1}) is established for $z$ on $\overline{\Omega}$, or for any $z\in\Omega$, our estimate \eqref{eqn.thm1} is only valid on compact subsets of $\Omega$.

In Section 5 below, we present numerical results which suggest that if $\Gamma$
is piecewise analytic without cusps, then the predicted order
$\mathcal{O}(1/n)$ in Theorem \ref{thm1} is sharp.

In contrast to \cite{Suetin1974} and related works for (complex) Jacobi matrices (see for instance \cite{Simon} or \cite[\S 3.5]{Bernd2001} and the references therein), our method of proof is not based on operator theory techniques, like determinants of the identity plus a trace class perturbation, and thus here, besides \eqref{assumption_epsilon}, we do not need any further assumption on the decay of the sequence $\{ \varepsilon_n \}$.

Our second result concerns the multiplication operator $H$, sometimes called Bergman shift,
which is defined by
\begin{equation} \label{multiplication}
     z p_n(z) = \sum_{j=0}^{n+1} p_j(z) H_{j,n} , \quad
     \mbox{with}~~~ H_{j,n} = \langle zp_n, p_j \rangle_{L^2(G)} .
\end{equation}
Notice that, by orthogonality, $H_{j,n}=0$ for $j>n+1$, and thus the infinite matrix $H$ has upper Hessenberg form. Using the Cauchy-Schwarz inequality it is easy to deduces that $H$ represents a bounded linear
operator in $\ell^2$, with $\| H \| \leq \sup\{ |z|: z \in G\}$, see also  \cite[\S 29]{Conway}.

\begin{theorem}\label{thm2}
   If $\Gamma$ is quasiconformal then there exists a constant $c(\Gamma)$ such that, for all $k\geq -1$ and $n \geq 0$,
   $$
          \Bigl| H_{n-k,n} - \sqrt{\frac{n+1}{n-k+1}} \psi_{-k} \Bigr|
          \leq c(\Gamma) \, (k+2) \,
         \max\Bigl(\varepsilon_{n-k-1},...,\varepsilon_{n},\varepsilon_{n+1}\Bigr).
   $$
\end{theorem}
Theorem~\ref{thm2} under the assuption \eqref{assumption_epsilon} yields immediately the following
estimate, for any fixed $k\geq -1$,
\begin{equation} \label{thm2_eqn}
          \Bigl| H_{n-k,n} - \sqrt{\frac{n+1}{n-k+1}} \psi_{-k} \Bigr|
          = \mathcal O(\frac{1}{n^\beta}), \quad n \to \infty .
\end{equation}
This estimate should be compared with the asymptotic results \cite[Theorem~2.1, Theorem~2.2]{Nikos2014}, where the same rate is given for $k=-1$ and $n \to \infty$, but for fixed $k \geq 0$ only the rate $\mathcal O(1/n^{\beta/2})_{n\to \infty}$ is obtained. Furthermore, if $\Gamma\in U(\rho)$, then Theorem~\ref{thm2}

yields $\mathcal O(\rho^{2n})$, for any fixed $k\geq -1$, improving thus by $\rho^n$ the rate established in \cite[Theorem~2.3]{Nikos2014} for $k\geq 0$.

In contrast to Theorem \ref{thm1}, the numerical results presented in \cite[\S4]{Nikos2014} for $\Gamma$ piecewise analytic without cusps indicate an experimental rate of convergence of $\mathcal{O}(1/n^2)$ for $k\in \{ -1,2\}$, rather than the rate $\mathcal{O}(1/n)$ predicted by Theorem~\ref{thm2}.

\bigskip

The paper is organized as follows. In \S 2 we derive, in Lemma~\ref{M=I-C*C},  a link between the Gram matrix
$(\langle \fa_j,\fa_k \rangle_{L^2(G)})_{j,k}$ and the Grunsky matrix.
Further, in Corollary~\ref{cor_Grunsky}
we suggest a new matrix formalism and obtain, in Theorem~\ref{thm_bounds}, estimates in terms of $\varepsilon_n$ for the coefficients in the change of basis from
$\{ p_0,p_1,...,p_n\}$ to $\{ \fa_0,\fa_1,...,\fa_n\}$, and vice-versa.
This allows us to present, in Remark~\ref{rem_equivalence}, different proofs of the relations  \eqref{asymp_leading} and \eqref{asymp_norm}.
We conclude \S 2 with a detailed discussion of an ellipse, where several of the above asymptotic results are shown to be sharp.

In \S 3 we gather some known results from the literature relating spectral and asymptotic properties of the Grunsky operator and the size of $\varepsilon_n$ to the smoothness of $\Gamma$. In particular, we conclude that the Grunsky operator for domains with corners cannot be compact.

The proofs of our main Theorems~\ref{thm1} and~\ref{thm2} are given in \S4.

Finally, in \S 5 we present a numerical experiment suggesting that the rate in Theorem~\ref{thm1}
is sharp for domains with corners.

\section{The Grunsky matrix and quasiconformal boundaries}

The aim of this section is to describe the size of the Fourier coefficients $R_{j,n}\in \mathbb C$ of $f_n$ in the orthonormal basis $\{ p_n\}$, that is,
\begin{equation} \label{basis}
    \fa_n(z) = \sum_{j=0}^n p_j(z) R_{j,n} , \quad  R_{j,n} = \langle \fa_n,p_j \rangle_{L^2(G)}.
\end{equation}
We note, in particular, that from \eqref{asymp_leading}
\begin{equation}\label{eq:Rnn}
R_{n,n}=\sqrt{\frac{n+1}{\pi}}\frac{\gamma^{n+1}}{\lambda_n}\in(0,1].
\end{equation}

It will be convenient to use matrix calculus. In what follows we denote by $e_j$, $j \geq 0$, the $j$-th canonical vector in $\mathbb C^n$ as well as in $\ell^2$, the size depending on the context. The action of a bounded linear operator $B:\ell^2\to \ell^2$ can be described through matrix products, where we identify $B$ with its infinite matrix $(\langle Be_k,e_j\rangle_{\ell^2})_{j,k=0,1,...}$.
We will also consider the operator $\Pi_n : \mathbb C^n \to \ell^2$ with matrix representation
$$
     \Pi_n = \left[\begin{array}{cc} I \\ 0 \end{array}\right] \in \mathbb C^{\infty\times n} ,
     \quad \mbox{with adjoint} \quad
     \Pi_n^* = \left[\begin{array}{cc} I & 0 \end{array}\right] \in \mathbb C^{n \times \infty} ,
$$
so that
$$
      B_n = \Pi_n^* B \Pi_n = (B_{j,k})_{j,k=0,1,...,n-1},
$$
gives the principal submatrix of order $n$ of $B$ (sometimes called the $n$-th finite section).

Below we will make use the relations
\begin{equation}\label{eq:matine}
    \| B_{n-1} \|\le \|B_n\|\le\|B\Pi_n\|\le\|B\|=\sup_{k}\|B_k\|,
\end{equation}
where $\| \cdot \|$
denotes the euclidian vector norm as well as the subordinate matrix (operator) norm, i.e.,
$\| B_n \|^2$ is the largest of the eigenvalues of the Hermitian matrix $B_n^*B_n$,

Then, it follows from
the Grunsky inequality \eqref{Grunsky_inequality} that the infinite Grunsky matrix $C=(C_{n,k})_{n,k=0,1,...}$ satisfies $\| C_n \|\leq \|C \Pi_n\|\leq 1$, for all $n$, and thus $\| C \|\leq 1$. 
We will also use frequently the identity
\begin{equation} \label{**}
   \varepsilon_n = \| C e_n \|^2,
\end{equation}
which follows easily from \eqref{norm_residual}.
Moreover, according to \cite[Theorems~9.12 and~9.13]{Pommerenke1975}, $\Gamma$ is quasiconformal if and only if  $\| C \|<1$ (and more precisely $\Gamma$ is $\kappa$-quasiconformal if and only if $\| C \|\leq \kappa$).

We recall now from \cite[\S 5]{Kuehnau1986} a geometry, where $\| C \|$ can be computed explicitly. Further geometries are discussed in \cite[\S 5 and \S 6]{Kuehnau1986}.

\begin{example}
    If $\Gamma$ is piecewise analytic with a corner of outer angle $\omega\pi$, $0<\omega<2$, then K\"uhnau showed, with the help of the Golusin inequality \cite[\S 3.2]{Pommerenke1975},  that $\| C \| \geq |1-\omega|$, with equality $\| C \| = |1-\omega|=2/m$, for all regular polygons with $m$ vertices. 
    \qed
\end{example}

In what follows we set $R_{j,n}=0$, for $j>n$, so that the infinite matrix $R$ is upper triangular,
with positive diagonal; see \eqref{eq:Rnn}.
As a consequence, $R_n$, $n=0,1,\ldots$, is invertible and any principal submatrix of order $k\leq n$ of $(R_n)^{-1}$ is given by $(R_k)^{-1}$.
Furthermore, since by changing the  basis in \eqref{basis} we get that
\begin{equation} \label{basis2}
    p_n(z) = \sum_{j=0}^n \fa_j(z) R^{-1}_{j,n},
\end{equation}
we can  write for $0 \leq j \leq k$, without ambiguity,
$R^{-1}_{j,k}=(R^{-1})_{j,k}=((R_k)^{-1})_{j,k}$, for the entries of the (possibly unbounded) upper triangular infinite matrix $R^{-1}$, the inverse of $R$.

We are now ready to make, in the next lemma, the link between the Grunsky matrix and Faber polynomials. Such a link has been probably first studied by K\"uhnau in \cite{Kuehnau1985}, though 
the identity \eqref{M=I-C*C_1} below can also be traced in the works of
Johnston \cite[Lemma~4.3.8]{Johnston} and Suetin \cite[p.\ 13]{Suetin1974}, without a direct
reference, however, to the relation with the Grunsky matrix. The same identity is mentioned without proof in \cite[Eqn.\ (2.9)]{Shen2009}.

\begin{lemma}\label{M=I-C*C}
   If $\Gamma$ has two-dimensional Lebesgue measure zero,
   then, for all $n,k\geq 0$, there holds
   \begin{align}
        \langle \fa_n,\fa_k \rangle_{L^2(G)} &= \delta_{n,k}  -
        \langle \fe_n,\fe_k \rangle_{L^2(\mathbb D^*)}
        \label{M=I-C*C_1}
        \\
        &= \delta_{n,k} - e_k^*C^* C e_n .
        \label{M=I-C*C_2}
   \end{align}
\end{lemma}
\begin{proof}
For the sake of completeness we reproduce here the idea of the proof given in  \cite{Kuehnau1985}.

For $r>1$ we consider the level set $G_r:= \overline{\mathbb C} \setminus \{ \psi(w):|w|\geq r \}$, which has analytic boundary $\partial G_r$. The proof below is based on the use of Green's formula
   $$
         \int_{G_r} f(u) \overline{g'(u)} dA(u) = \frac{1}{2i} \int_{\partial G_r} f(u) \overline{g(u)} du ,
   $$
   for any $f$, $g$ analytic on $\overline{G_r}$.

    By comparing like powers of $v$ in \eqref{Grunsky_coefficients} and using \eqref{normalized_Grunsky_coefficients}, we have for $|w|>1$ that
   \begin{equation} \label{primitive_Faber}
          \frac{F_{k+1}(\psi(w))}{\sqrt{\pi(k+1)}} =
          \frac{w^{k+1}}{\sqrt{\pi(k+1)}} + \sum_{\ell=0}^\infty
          \frac{w^{-(\ell+1)}}{\sqrt{\pi(\ell+1)}} C_{\ell,k}.
   \end{equation}
   (We note that taking derivatives in \eqref{primitive_Faber} leads to
   \eqref{faber_residual}.)
   We thus
   obtain by applying Green's formula
   \begin{align*}
      \int_{G_r} \fa_n(z) \overline{\fa_k(z)} dA(z)
      &= \frac{1}{2i} \int_{\partial G_r}
      \frac{F_{n+1}'(z)}{\sqrt{\pi(n+1)}} \overline{\frac{F_{k+1}(z)}{\sqrt{\pi(k+1)}}} dz
      \\&= \frac{1}{2i} \int_{|w|=r}
      \frac{\psi'(w)F_{n+1}'(\psi(w))}{\sqrt{\pi(n+1)}} \overline{\frac{F_{k+1}(\psi(w))}{\sqrt{\pi(k+1)}}} dw.
   \end{align*}

   Next, inserting \eqref{faber_residual} and \eqref{primitive_Faber} in the latter integral we get the expression
   \begin{eqnarray*} &&
      \frac{1}{2i} \int_{|w|=r}
      \Bigl( \sqrt{\frac{n+1}{\pi}} w^n
    - \sum_{\ell=0}^\infty \sqrt{\frac{\ell+1}{\pi}}w^{-(\ell+2)} C_{\ell,n}
      \Bigr)
      \overline{\Bigl(
                \frac{w^{k+1}}{\sqrt{\pi(k+1)}} + \sum_{j=0}^\infty
          \frac{w^{-(j+1)}}{\sqrt{\pi(j+1)}} C_{j,k}
      \Bigr)} dw.
   \end{eqnarray*}
   Here the sums are uniformly convergent in $|w|=r$ and we can thus integrate term by term. Since,
   $$
          \frac{1}{2\pi i} \int_{|w|=r} w^j \overline{w^{\ell+1}} dw =
          \frac{r^{2\ell+2} }{2\pi i} \int_{|w|=r} w^j w^{-\ell-1} dw = r^{2\ell+2} \, \delta_{j,\ell},
   $$
   for $j,\ell \in \mathbb Z$,
   we therefore arrive at
   \begin{equation} \label{eqn.level}
      \langle \fa_n, \fa_k\rangle_{L^2(G_r)}
      = \delta_{n,k} r^{2k+2} - \sum_{\ell=0}^\infty \overline{C_{\ell,k}} C_{\ell,n} r^{-2\ell-2}.
   \end{equation}

   Clearly, the polynomials $\fa_n$ and $\fa_k$ are continuous in $\mathbb C$ and the closure $G \cup \Gamma$ of $G$ is the intersection of the monotone family $G_r$, for $r>1$.
   Therefore, by using our assumption on $\Gamma$, we conclude that
   $$
      \lim_{r\to 1+} \langle \fa_n, \fa_k\rangle_{L^2(G_r)}
      =       \langle \fa_n, \fa_k\rangle_{L^2(G \cup \Gamma)}
      =       \langle \fa_n, \fa_k\rangle_{L^2(G)}.
   $$

   From the inequality $\| C \|\leq 1$ we have that $Ce_k = (C_{j,k})_j \in \ell^2$, for all $k\geq 0$
    and hence $(\overline{C_{j,k}} C_{j,n})_j \in \ell^1$,
   implying that the term on the right-hand side of \eqref{eqn.level} has a (finite) limit for $r\to 1+$, given by
   $\delta_{n,k} - e_k^* C^* C e_n$. Moreover, by integrating \eqref{faber_residual} over $r\mathbb D^*$ and working as above we obtain
   $$
        \lim_{r\to 1+} \langle \fe_n, \fe_k\rangle_{L^2(r \mathbb D^*)}
        = \lim_{r\to 1+} \sum_{\ell=0}^\infty \overline{C_{\ell,k}} C_{\ell,n} r^{-2\ell-2} = \sum_{\ell=0}^\infty \overline{C_{\ell,k}} C_{\ell,n}
        = \langle \fe_n, \fe_k\rangle_{L^2(\mathbb D^*)},
   $$
   which, in view of \eqref{eqn.level}, yields both \eqref{M=I-C*C_1} and \eqref{M=I-C*C_2}.
\end{proof}

It should be noted that in the proof we did not use any assumptions on the boundary $\Gamma$ of $G$, other that it should have zero Lebesgue measure. In particular, we did not require that $\Gamma$ is a Jordan curve.

Next we apply Lemma~\ref{M=I-C*C} to quasiconformal $\Gamma$. Then,
from \eqref{basis} and the orthonormality of Bergman polynomials, we find that
\begin{equation} \label{eq.cholesky}
    \langle \fa_n,\fa_k \rangle_{L^2(G)} = \sum_{j=0}^{\min(n,k)} \overline{R_{j,k}} R_{j,n}.
\end{equation}

Recalling that $\| B \| = \| B^* \|=\sqrt{\| B^* B \|}$, 
the following result is an easy consequence of Lemma~\ref{M=I-C*C}.
\begin{corollary}\label{cor_Grunsky}
    If $\Gamma$ is quasiconformal, then the infinite Gram matrix
    $M=( \langle \fa_n,\fa_k \rangle_{L^2(G)} )_{k,n=0,1,...}$ can be decomposed as
    \begin{equation} \label{M=I-C*C_3}
         M = R^* R = I - C^*C,
    \end{equation}
    where $M$ and $R$ represent two bounded linear operators on $\ell^2$ with norms $\leq 1$.
     Furthermore, both $M$ and $R$ are boundedly invertible, with
     $\| M^{-1} \| = \| R^{-1} \|^2 \leq (1-\| C \|^2)^{-1}$.
\end{corollary}
\begin{proof}
    Equality for the entries of the matrices occurring in \eqref{M=I-C*C_3} follows from Lemma~\ref{M=I-C*C} and \eqref{eq.cholesky}. It remains to show that the involved matrices represent bounded linear operators on $\ell^2$.

    From \eqref{eq.cholesky} we obtain the Cholesky decomposition
    \begin{equation}\label{eq:Mn=Rn*Rn}
     M_n=R_n^*R_n,
    \end{equation}
    which, in view of the fact that $R_n$ is invertible, implies that $M_n$ is Hermitian positive definite. Also, taking finite sections in \eqref{M=I-C*C_3}, we get that
    $$
         I - M_n=\Pi_n^* C^*C \Pi_n
    $$
    is Hermitian positive semi-definite, and hence the eigenvalues of $M_n$ are elements of $(0,1]$,
    that is $\| M_n \|=\| R_n\|^2 \leq 1$, and thus $\| M \|=\| R \|^2 \leq 1$ by \eqref{eq:matine}, as claimed above.

    Our assumption on $\Gamma$ implies than $\| C^* C \| = \| C \|^2<1$, and therefore a Neumann series argument in $\ell^2$ shows that $M$ is boundedly invertible, with $\| M^{-1} \|\leq 1/(1-\| C \|^2)$.

    Finally, the existence of $R^{-1}$ has been established above, as a result of the representation
    \eqref{basis2}. Hence, from \eqref{M=I-C*C_3}, $R^{-1}(R^{-1})^*=M^{-1}$, leading to the required relation $\|M^{-1}\|=\|R^{-1}\|^2$.
\end{proof}


Our next result provides estimates for the coefficients $R_{j,k}$ and $R^{-1}_{j,k}$ in
(\ref{basis}) and (\ref{basis2}).
\begin{theorem}\label{thm_bounds}
   If $\Gamma$ is quasiconformal, then for all $j\geq 0$,
   \begin{equation} \label{bound_R_row}
       \max\Bigl( \| e_j^* (I-R^*) \| , \| e_j^* (R^{-1}-I) \|
       \Bigr) \leq \| e_j^* (R^{-1}-R^*) \| \leq \sqrt{\varepsilon_j \, \frac{\| C \|^2 }{1-\| C \|^2}}.
   \end{equation}
   Furthermore, for all $0 \leq j\leq n$,
   \begin{equation}  \label{bound_R_entry}
        \max \Bigl( | R_{j,n}-\delta_{j,n}| , | R^{-1}_{j,n} -\delta_{j,n}| \Bigr)
        \leq \frac{\sqrt{\varepsilon_j \, \varepsilon_n} }{1-\| C \|^2}.
   \end{equation}
\end{theorem}
\begin{proof}
   We first observe that from Lemma~\ref{M=I-C*C} with $k=n$,  \eqref{basis} and \eqref{**},
   we have
   \begin{equation} \label{R0}
      \| C e_n\|^2 = \varepsilon_n = 1 - \| f_n \|_{L^2(G)}^2 = 1 - \sum_{j=0}^n |R_{j,n}|^2,
   \end{equation}
    and recall from \eqref{eq:Rnn} that $R_{n,n}\in(0,1]$. This, in particular, implies that
   \begin{equation} \label{R0bis}
         0 \leq \max (1-R_{n,n},\frac{1}{R_{n,n}}-1 ) \leq \frac{1}{R_{n,n}}-R_{n,n} .
   \end{equation}

   Next we establish the norm estimates in (\ref{bound_R_row}). For this we note that by
   Corollary~\ref{cor_Grunsky},
   \begin{equation}\label{eq:R-R}
         R^{-1} - R^* = (I-R^*R)R^{-1} = C^*C R^{-1}
   \end{equation}
   and hence, for any $j \geq 0$,
   $$
        \| e_j^* (R^{-1} - R^*) \|^2 =
        \| (Ce_j)^* C R^{-1} \|^2
        \leq \varepsilon_j \| C R^{-1} \|^2 \leq  \frac{\varepsilon_j \| C \|^2}{1-\| C \|^2}.
   $$
   This yields the second inequality in \eqref{bound_R_row}. Taking into account \eqref{R0bis} in conjunction with the obvious relation
   $$
              \| e_j^* (R^{-1} - R^*) \|^2 =
        \sum_{n=j+1}^\infty |R^{-1}_{j,n}|^2 + | \frac{1}{R_{j,j}}-R_{j,j} |^2
        + \sum_{n=0}^{j-1} |R_{n,j}|^2,
   $$ we arrive at the first inequality in \eqref{bound_R_row}.

   In order to obtain estimates for each entry of $R^{-1}-I$ and $I-R^*$, we argue in terms of the entries of $R^{-1}-R^*$. To this end, we recall that $R_{n+1}$ and $R^*_{n+1}$ are both bounded and invertible and thus so does $M_{n+1}$, in view of \eqref{eq:Mn=Rn*Rn}. Furthermore,
   from Corollary~\ref{cor_Grunsky},
   \begin{equation}\label{eq:MP}
   M_{n+1}= \Pi_{n+1}^* M \Pi_{n+1} =  I - \Pi_{n+1}^* C^* C \Pi_{n+1},
   \end{equation}
   where $\| C \Pi_{n+1} \| \leq \| C \| < 1$ by (\ref{eq:matine}). 
   
   According to the upper triangular structure of $R_{n+1}$, we find that
   $$
        C R^{-1} e_n = C \Pi_{n+1} R_{n+1}^{-1} e_n
        = C \Pi_{n+1} M_{n+1}^{-1} R_{n+1}^* e_n = R_{n,n} \, C \Pi_{n+1} M_{n+1}^{-1} e_n.
   $$
   Hence, by using a standard Neumann series argument we obtain from \eqref{eq:R-R} and \eqref{eq:MP},
   for any $n,j\geq 0$, that
   \begin{align*}
        e_j^* (R^{-1} - R^*) e_n
        &= R_{n,n}   (Ce_j)^*
        C \Pi_{n+1} M_{n+1}^{-1} e_n
\\
        &= R_{n,n}   (Ce_j)^*
        C \Pi_{n+1} \sum_{\ell=0}^\infty (\Pi_{n+1}^*C^*C \Pi_{n+1})^\ell e_n
\\
        &= R_{n,n}   (Ce_j)^*
        \sum_{\ell=0}^\infty (C \Pi_{n+1} \Pi_{n+1}^*C^*)^\ell C e_n
\\
        &= R_{n,n}   (Ce_j)^*(I-C \Pi_{n+1}\Pi_{n+1}^*C^*)^{-1}C e_n .
   \end{align*}
   Now, taking norms and using \eqref{eq:Rnn}, \eqref{eq:matine} and \eqref{**}, we get
   \begin{equation} \label{aim77}
        | e_j^* (R^{-1} - R^*) e_n | \leq R_{n,n} \frac{\sqrt{\varepsilon_j \varepsilon_n}}{1-\| C\|^2}\leq \frac{\sqrt{\varepsilon_j \varepsilon_n}}{1-\| C\|^2}.
   \end{equation}
   and recalling that $R^{-1}$ and $R$ are upper triangular, we conclude, for $j<n$, that
   \begin{equation*}
   R_{j,n}^{-1}= e_j^* (R^{-1} - R^*) e_n\quad \textup{and}\quad  R_{j,n}= e_n^* (R^{-1} - R^*) e_j
   \end{equation*}

   Thus \eqref{bound_R_entry}, with $0\leq j<n$, follows immediately from \eqref{aim77}, whereas
   for the diagonal entries $j=n$, we need to combine \eqref{aim77} with \eqref{R0bis}.
\end{proof}

\begin{remark}\label{rem_equivalence}
   With the estimates in Thorem~\ref{thm_bounds} we can produce more direct proofs of \eqref{asymp_leading} and \eqref{asymp_norm}.

   Since $R_{n,n}^{-1}=1/R_{n,n}$, we obtain from \eqref{eq:Rnn}, \eqref{R0} and \eqref{aim77}, for $j=n$,
   $$
       \varepsilon_n \leq 1 - R_{n,n}^2  \leq \frac{1}{R_{n,n}} - R_{n,n} \leq \frac{\varepsilon_n}{1-\| C \|^2},
   $$
   which yields \eqref{asymp_leading}.

   Next, from \eqref{basis} we have
   $$
       p_n(z) - \fa_n(z) = (p_0(z),...,p_n(z))( I- R_{n+1})e_n.
   $$
   Thus, using \eqref{bound_R_row} we see that
   \begin{equation} \label{asymp_norm2}
      \| \fa_n - R_{n,n} p_n \|_{L^2(G)}^2
      = \sum_{j=0}^{n-1} | R_{j,n} |^2 \leq \|  ( I- R_{n+1})e_n \|^2
    \leq \frac{\varepsilon_n \, \| C \|^2}{1-\| C \|^2},
   \end{equation}
   and the result (\ref{asymp_norm}) follows because
   \begin{equation*} \label{asymp_norm2-2}
      \| \fa_n - p_n \|_{L^2(G)}^2=\|  ( I- R_{n+1})e_n \|^2.
   \end{equation*}
  \qed
\end{remark}

   It is interesting to note that under the assumption $\Gamma$ is quasiconformal,
   (whence $\| C \| \leq \kappa$, for some $\kappa<1$), a similar inequality
  to (\ref{asymp_norm2}), namely
   $$
      \| \fa_n - R_{n,n} p_n \|_{L^2(G)}^2 \leq \frac{\varepsilon_n \, \lambda^2}{1-\lambda^2},
   $$
   with $\lambda\in [0,1)$, has been established in \cite[Theorem~2.1]{Nikos2013} by
   using a quasiconformal reflection argument.

\subsection{An example}\label{example_ellipse}
We conclude this section with a detailed study of asymptotics when $\Gamma$ is an ellipse. In this case \eqref{faber_residual} is very simple, allowing us to give explicitly the matrix $C$, and thus to illustrate many of the results presented so far.

    Let $G$ be the interior of the ellipse with semi-axes $(r\pm r^{-1})/2$, for some $r>1$. Then, the corresponding conformal map
    $\psi:\mathbb D^*\mapsto \Omega$ is given by $$
          \psi(w)=\frac{1}{2}( rw + \frac{1}{rw})
    $$
    and $\gamma=2/r$. Clearly, $\psi$ has an analytic continuation for $|w|>0$, which is however only univalent for $|w|>1/r\in (0,1)$, as the equation $\psi'(w)=0$ reveals. Hence $\Gamma\in U(1/r)$ and thus $\varepsilon_n=\mathcal O(r^{-2n})$ by \eqref{summarize}.
    We show that, for ellipses, this estimate can be improved.

    Inserting the formula for $\psi(w)$ in \eqref{Grunsky_coefficients}, we deduce the following expression  for the associated Faber polynomials
    $$
          F_{n}(\psi(w)) = w^{n} + r^{-2n} w^{-n}
    $$
    for $n \geq 1$ and $w\in \mathbb D^*$, see, e.g., \cite[p.~134]{SmLe68}. Thus, for $w\in \mathbb D^*$,
    \begin{equation} \label{cheby1}
        r_n(w)=\psi'(w)f_n(\psi(w)) - \sqrt{\frac{n+1}{\pi}} w^n = -
        \sqrt{\frac{n+1}{\pi}} \frac{r^{-2n-2}}{w^{n+2}}.
    \end{equation}
    Comparing with \eqref{faber_residual} and \eqref{norm_residual} we conclude that
    \begin{equation} \label{cheby2}
         C=\textup{diag}(r^{-2},r^{-4},...) ,
         \quad \| C \|=r^{-2} ,
         \quad
         \varepsilon_n = \| Ce_n\|^2 = \| r_n\|^2_{L^2(\mathbb D^*)} = r^{-4n-4},
    \end{equation}
     cf. \cite[\S 3]{Kuehnau1986}.

     It follows therefore from Corollary~\ref{cor_Grunsky} that, like $C$, also $M,R,$ and $R^{-1}$ are
     diagonal, with $R^2_{n,n}={1-\varepsilon_n}$. Since $0< R_{n,n}\le 1$, this leads to the chain of relations
    $$
          \frac{\varepsilon_n}{2} \leq \| e_n^* (I-R^*) \| = 1 - R_{n,n} \leq \frac{1}{R_{n,n}} - 1
          =\| e_n^* (I-R^{-1}) \|
          \leq \frac{\varepsilon_n}{1-\| C \|^2},
    $$
     which, in conjunction with (\ref{cheby2}) show that the estimate \eqref{bound_R_entry}, for $j=n$, is sharp, up to some constant,  whereas \eqref{bound_R_row} is not.

    In view of \eqref{basis}, a diagonal $R$ also implies the well-known fact \cite[p.~547]{Henrici} that the $n$-th Bergman polynomial for ellipses is proportional to $F_{n+1}'$, and more precisely
    \begin{equation} \label{cheby3}
         \fa_n(z) = R_{n,n} p_n(z) , \quad
         R_{n,n} = \sqrt{1-\varepsilon_n} = \sqrt{\frac{n+1}{\pi} \frac{\gamma^{2n+2}}{\lambda_n^2}} .
    \end{equation}
    Thus we obtain equality in \eqref{asymp_leading}, i.e.,
    $$
          1 - \frac{n+1}{\pi} \frac{\gamma^{2n+2}}{\lambda_n^2}=\varepsilon_n.
    $$
    On the other hand, both \eqref{asymp_norm} and \eqref{asymp_outside} are not sharp, since \eqref{cheby3} implies
    $$
          \| f_n - p_n \|^2_{L^2(G)} = (1-R_{n,n})^2 = \varepsilon_n^2/4+\mathcal O(\varepsilon_n^3)
    $$
    and for $z\in \Omega$,
    $$
          \frac{\fa_n(z)}{p_n(z)} = R_{n,n} = 1 + \mathcal O(\varepsilon_n).
    $$

    Finally, it turns out that the exponential term $\sqrt{\varepsilon_n}=r^{-2n-2}$ in our Theorem~\ref{thm1} cannot be improved for ellipses. Indeed, for $z\in \Omega$ we find using \eqref{cheby1}--\eqref{cheby3} and the fact that $|\phi(z)|>1$,
    \begin{align*}
         \sqrt{\frac{\pi}{n+1}} \frac{p_n(z)}{\phi'(z)\phi(z)^n} - 1
         &=
         \sqrt{\frac{\pi}{n+1}} \frac{\fa_n(z)}{R_{n,n} \phi'(z)\phi(z)^n} - 1
\\&= \frac{1}{R_{n,n}} \Bigl( 1 - \frac{r^{-2n-2}}{\phi(z)^{2n+2}} \Bigr) - 1 \\
&= \mathcal O(\varepsilon_n + \frac{\sqrt{\varepsilon_n}}{|\phi(z)|^{2n}})
= \mathcal O(\frac{\sqrt{\varepsilon_n}}{n^{\beta/2}})
    \end{align*}
    for all $\beta>0$, uniformly on compact subsets $K$ of $\Omega$.
    Clearly, in 
    the right-hand side we may not replace $\sqrt{\varepsilon_n} = r^{-2n-2}$ by $\tau^{-2n}$, for some
    $\tau>r$ independent of $K$.

\section{The Grunsky matrix and smooth boundaries}

All asymptotic results obtained so far depend on whether
$\varepsilon_n=\| Ce_n\|^2$ tends to zero, and with what speed. We have already seen in \eqref{summarize} three different examples of smoothness of $\Gamma$, for which we have further information on $\varepsilon_n$, due to the works of Carleman, Suetin and one of the authors \cite{Carleman1923,Suetin1974,Nikos2013}. The aim of this section is to gather other known results in the literature about spectral properties of the Grunsky matrix, and relate them to the smoothness of $\Gamma$.

Before going further, it will be useful to recall the generating function for the Grunsky matrix, as well as some basic facts on compact operators.
For $w \in \mathbb D^*$ we define the infinite vector
\begin{equation} \label{y(w)}
      y(w) = \Bigl(  \sqrt{\frac{\ell+1}{\pi}}\frac{1}{w^{\ell+2}} \Bigr)_{\ell=0,1,...} \in \ell^2, \quad
      \mbox{with} \quad
      \| y(w)\|^2 = \frac{1}{\pi (|w|^2-1)^2},
\end{equation}
by elementary computations.
The way that $C$ depends on $\psi$ is quite involved. To see this,
for distinct $w,v\in \mathbb D^*$, we differentiate \eqref{Grunsky_coefficients} with respect to $v$ and $w$ and employ \eqref{normalized_Grunsky_coefficients} to arrive at
the generating function \begin{align*}
        \frac{1}{\pi}\Bigl( \frac{\psi'(w)\psi'(v)}{(\psi(v)-\psi(w))^2} - \frac{1}{(v-w)^2}\Bigr)
        &= - y(v)^T C y(w) = - \langle C y(w),y(\overline v) \rangle_{\ell^2}
        \\&= \sum_{n=0}^\infty \sqrt{\frac{n+1}{\pi}}\frac{\fe_{n}(w)}{v^{n+2}},
\end{align*}
where in the last equality we made use of \eqref{faber_residual}.

For $w=v$ we get a link with the Schwarzian derivative of $\psi$ in $\Omega$,
see, e.g., \cite[pp.\ 2128--2129]{Shen2009},
\begin{equation} \label{Sch}
     S_\psi(w):= 
     \frac{\psi'''(w)}{\psi'(w)} - \frac{3}{2}\Bigl( \frac{\psi''(w)}{\psi'(w)}\Bigr)^2
     = - {6\pi} y(w)^T C y(w),
\end{equation}
a well-known quantity to measure smoothness of $\psi$ or $\Gamma$ \cite[\S 11.2 and \S 1.5]{Pommerenke1992}. In particular, using \eqref{y(w)} we find that, for $w\in \mathbb D^*$,
$$
     \frac{(|w|^2-1)^2}{6} | S_\psi(w) | = \frac{|y(w)^T C y(w)|}{\| y(w) \|^2} \leq \| C \|.
$$

Next we recall a number of facts from operator theory regarding compact operators in $l^2$.

An operator $B$ acting on $\ell^2$ is called compact if the image under $B$ of any bounded sequence in $\ell^2$ contains a subsequence which is convergent in norm \cite[\S III.4.1]{Kato1980}. In particular, we deduce that $Be_n \to 0$, as $n\to\infty$.
Also, $B$ is compact if and only if the sequence of the singular values of $B$ tends to $0$ \cite[\S V.2.3]{Kato1980}.

$B$ is called compact of Schatten class of index $p$ (trace class if $p=1$, Hilbert-Schmidt class if $p=2$) if the sequence of singular values of $B$ is an element of $\ell^p$
\cite[\S X.1.3]{Kato1980}. In particular $B^*B$ is of trace class if and only if $B$ is of Hilbert-Schmidt class, i.e., $\sum_{n=0}^\infty \| B e_n \|^2<\infty$.

Finally, the product of a compact operator with a bounded operator is compact 
\cite[Theorem III.4.8]{Kato1980}, and so is the adjoint of a compact operator 
\cite[Theorem III.4.10]{Kato1980}.
Similar properties hold for operators of Schatten class of index $p$ \cite[\S X.1.3]{Kato1980}.

In view of the above, Corollary~\ref{cor_Grunsky} and Theorem~\ref{thm_bounds} lead to the following.

\begin{corollary}\label{cor_Schatten}
    $C$ is compact (of $p$-Schatten class) if and only if $I-M$ and/or $I-M^{-1}$ are compact
    (of $p/2$-Schatten class). Furthermore, if $C$ is Hilbert-Schmidt then so are both $I-R$ and $I-R^{-1}$.
\end{corollary}

\begin{remark}
    Suetin made use in \cite[\S 1]{Suetin1974} of the so-called {\em method of normal moments}, where it is required that
    $$
         \sum_{j,k=0}^\infty | e_j^* C^*C e_k |^2 < \infty,
    $$
    or, in operator terms, $C^*C$ is of Hilbert-Schmidt class (or equivalently
    $C$ is of $4$-Schatten class). To ensure this, he derived in \cite[Lemma~1.5]{Suetin1974}
     and subsequent considerations an upper bound for $| e_j^* C^*C e_k |$ in terms of the smoothness properties of $\Gamma$. More precisely, he showed that there exists a constant $c(\Gamma)$ such that, for all $j,k\geq 0$,
    $$
        | e_j^* C^*C e_k | \leq \frac{c(\Gamma)}{(j+1)^{\beta/2}(k+1)^{\beta/2}},
    $$
    provided that $\Gamma\in \mathcal C(p+1,\alpha)$ with $p\geq 0$, and
    $\beta=2p+2\alpha>1$. For $j=k=n$ this leads to the estimate $\varepsilon_n=\| C e_n \|^2 = \mathcal O(1/n^\beta)$. In other words, Suetin only dealt with cases when $C$ belongs to the Hilbert-Schmidt class, or equivalently, when $I-M=C^*C$ is of trace class.

    It is not evident whether Suetin was aware of the
    classical theory of determinants of trace class perturbations of the identity as given, for example,
    in \cite[\S XXI.9]{Dunford}. Nevertheless, in \cite[Lemma~1.6 and~1.7]{Suetin1974} this theory is
    tacitly used
    in order to compute limits and upper bounds for various determinants, in particular, to
    show the existence of the following limits
    $$
      \lim_{n\to \infty} \prod_{j=0}^{n-1} R_{j,j}^2 = \lim_{n\to \infty} \det(R_n)^2 =\lim_{n\to \infty} \det(M_n) = \det(M)>0,
    $$
    and to obtain upper bounds for the numbers
    $$
         \pm R_{n,n}^{-1} R_{j,n}^{-1} = \det \left[\begin{array}{cc}
             I_{n+1} & (R_{n+1}^{-1})^* e_n \\
             e_j^* R_{n+1}^{-1} & 0
               \end{array}\right]
         = \frac{\det \left[\begin{array}{cc}
                  M_{n+1} & e_n \\ e_j^* & 0
            \end{array}\right]      }{\det(M_{n+1})} = - e_j^* M_{n+1}^{-1} e_n .
    $$
\qed
\end{remark}
A compact Grunsky matrix $C$ can be equivalently characterized by the fact that $\Gamma$ is asymptotically conformal \cite[Theorem~5.1]{Shen2009},
which by \cite[Theorem~11.1]{Pommerenke1992} 
is equivalent to the property $(|w|^2-1)^2 S_\psi(w)\to 0$ for $|w|\to 1+$.

We recall from above that, in any of these cases, $\varepsilon_n\to 0$, $n \to \infty$. The converse is however not true, as the following result shows.
\begin{proposition}\label{co2}
  The following statements hold.
  \begin{itemize}
  \item[(a)]  If $\Gamma$ has a corner (with angle different than $\pi$), then the corresponding Grunsky operator $C$ is not compact.
  \item[(b)]
  If the Grunky operator associated with $\Gamma$ is compact then $\Gamma$ is a quasiconformal curve.
  \item[(c)]
  $\Gamma$ is an analytic Jordan curve if and only if there exists $\rho\in [0,1)$ such $\varepsilon_n=\mathcal O(\rho^{2n})$.
   \item[(d)]
   There exists a quasiconformal curve $\Gamma$ such that, for infinitely many $n\in\mathbb{N}$,
   \begin{equation*}
   \varepsilon_n\ge \frac{\gamma^2}{(n+1)^{1-1/25}}.
   \end{equation*}
  \end{itemize}
\end{proposition}
\begin{proof}
   Part (a) follows by the observation that asymptotically conformal does not allow corners by \cite[\S 11.2]{Pommerenke1992}.

   To show part (b), recall from above that if the Grunsky operator is compact then $\Gamma$ is asymptotically conformal. This latter property implies that $\Gamma$ is also quasiconformal, see \cite[\S 11.2]{Pommerenke1992}.

   The implication $\Longrightarrow$
   in part (c) is included in
   \eqref{summarize}. In order to show the reciprocal, suppose that $\varepsilon_n = \mathcal O(\rho^{2n})$ for some $\rho\in [0,1)$. Then, there exists a constant $c$ such that
   \begin{equation} \label{C0ell}
       |C_{\ell,0}|=|C_{0,\ell}| \leq \| Ce_\ell \| = \sqrt{\varepsilon_\ell} \leq c \, \rho^\ell
   \end{equation}
   for all $\ell\geq 0$. Hence, we get from \eqref{faber} and \eqref{faber_residual} for $n=0$ and
   $|w|>1$ that
   \begin{equation} \label{eq:derivative}
      \fe_0(w)=\psi'(w)\frac{\gamma}{\sqrt{\pi}} - \frac{1}{\sqrt{\pi}}
    = - \sum_{\ell=0}^\infty \sqrt{\frac{\ell+1}{\pi}}w^{-(\ell+2)} C_{\ell,0},
   \end{equation}
   and, in view of \eqref{C0ell}, we conclude that $\psi$ has an analytic continuation across the unit circle into $\mathbb D$. Thus, $\Gamma$ is an analytic image of the unit circle. Therefore, around any $w_0$ of modulus $1$ the map $\psi$ can be represented by a Taylor series expansion of the form
   $$
       \psi(w)=\psi(w_0) + a_1 (w-w_0)+a_2(w-w_0)^2+\cdots .
   $$
   If $a_1=\psi'(w_0)=0$, then necessarily $a_2\neq 0$ since $\psi$ is univalent in $\mathbb D^*$.
   This shows that $w_0$ would be mapped by $\psi$ onto an exterior pointing cusp on $\Gamma$.
   Our assumption on $\varepsilon_n$ also implies that $C$ is of Hilbert-Schmidt class and thus is compact. From part (a) we conclude that $\Gamma$ has no corners, and thus $\psi'(w)\neq 0$ for all $|w|=1$. This implies the univalence of $\psi$ in a neighborhood of $|w|=1$, leading to the conclusion that $\Gamma$ is an analytic Jordan curve.

   By comparing coefficients in the series expansion \eqref{eq:derivative} to that of $\psi(w)$ given in
   the introduction, we obtain the well-known relation, for $n\ge 0$,
   \begin{equation*}
   \sqrt{n+1}b_{n+1,0}=C_{n,0}=\gamma\sqrt{n+1}|\psi_{n+1}|;
   \end{equation*}
   cf. \cite[\S3.1]{Pommerenke1975}.
   This leads to the inequality\footnote{Compare with \cite[Theorem~2.2]{Nikos2013}.}
   \begin{equation}
       \varepsilon_n \geq |C_{0,n}|^2 = \gamma^2 (n+1) | \psi_{-n-1}|^2,
    \end{equation}
   for all $n\ge 0$. It suffices, therefore, for part (d) to recall the inequality
   $$
   n|\psi_{-n}|>n^{1/50},
   $$
   which holds, for infinitely many $n\in\mathbb{N}$,
   for the quasiconformal curve constructed by Clunie in \cite{Cl59},
   see also \cite[\S4.2]{Ga99} and \cite[p.~63]{Nikos2013}
\end{proof}

As an inference of Proposition~\ref{co2}(a), we remark that if $\Gamma$ is
piecewise analytic and has corners of angles different than $0$, $\pi$ or $2\pi$, then its Grunsky operator $C$ cannot be compact, though in this case $\varepsilon_n=\mathcal O(1/n)$ by \eqref{summarize}.

An assertion similar to Proposition~\ref{co2}(c) has been shown to hold in \cite[Theorem~2.4]{Nikos2014}, under the additional assumption, however, that $\Gamma$ has no zero interior angles.

With respect to Proposition~\ref{co2}(d),
we do not know whether the sole assumption of $\Gamma$ being $\kappa$-quasiconformal (or $\| C \| \leq \kappa<1$) is sufficient to imply that $\varepsilon_n \to 0$, $n \to \infty$.

\begin{figure}[t]
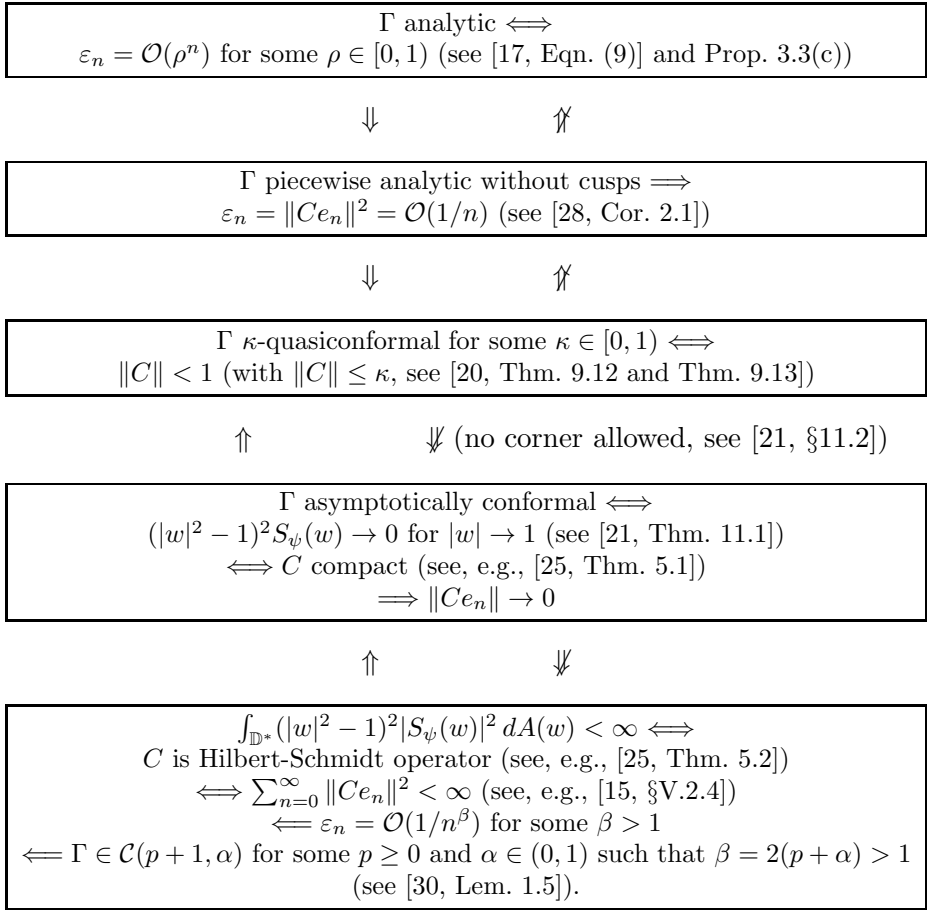

\begin{center}
   \fbox{\small\begin{minipage}[c]{12cm}\begin{center}
       $\Gamma$ analytic $\Longleftrightarrow$\\
       $\varepsilon_n = \mathcal O(\rho^n)$ for some $\rho\in [0,1)$ (see \cite[Eqn.\ (9)]{Kuehnau1986} and Prop.~\ref{co2}(c))
      \end{center}\end{minipage}}
   \\[10pt]
       $\Downarrow \qquad \qquad \qquad \not\Uparrow$
   \\[10pt]
   \fbox{\small\begin{minipage}[c]{12cm}\begin{center}
       $\Gamma$ piecewise analytic without cusps $\Longrightarrow$\\
       $\varepsilon_n = \| Ce_n\|^2 = \mathcal O(1/n)$ (see \cite[Cor.~2.1]{Nikos2013})
   \end{center}\end{minipage}}
   \\[10pt]
       $\Downarrow \qquad \qquad \qquad \not\Uparrow$
   \\[10pt]
   \fbox{\small\begin{minipage}[c]{12cm}\begin{center}
       $\Gamma$ $\kappa$-quasiconformal for some $\kappa\in [0,1)$ $\Longleftrightarrow$\\
       $\| C \|<1$ (with $\| C \| \leq \kappa$, see \cite[Thm.~9.12 and Thm.~9.13]{Pommerenke1975})
   \end{center}\end{minipage}}
   \\[10pt]
       $~ \qquad \qquad \qquad  \Uparrow \qquad \qquad \qquad \not\Downarrow$ (no corner allowed, see \cite[\S 11.2]{Pommerenke1992})
   \\[10pt]
   \fbox{\small\begin{minipage}[c]{12cm}\begin{center}
       $\Gamma$  asymptotically conformal $\Longleftrightarrow$\\
       $(|w|^2-1)^2 S_\psi(w)\to 0$ for $|w|\to 1$  (see \cite[Thm.~11.1]{Pommerenke1992}) \\
       $\Longleftrightarrow C$ compact (see, e.g., \cite[Thm.~5.1]{Shen2009})\\
       $\Longrightarrow \|C e_n\|\to 0$
   \end{center}\end{minipage}}
   \\[10pt]
       $\Uparrow \qquad \qquad \qquad \not\Downarrow$
   \\[10pt]
   \fbox{\small\begin{minipage}[c]{12cm}\begin{center}
       $\int_{\mathbb D^*} (|w|^2-1)^2 |S_\psi(w)|^2 \, dA(w) < \infty$
       $\Longleftrightarrow$ \\
       $C$ is Hilbert-Schmidt operator (see, e.g., \cite[Thm.~5.2]{Shen2009})\\
       $\Longleftrightarrow \sum_{n=0}^\infty \|C e_n\|^2 <\infty$
       (see, e.g., \cite[\S V.2.4]{Kato1980})\\
       $\Longleftarrow \varepsilon_n=\mathcal O(1/n^\beta)$ for some $\beta>1$
       \\
       $\Longleftarrow \Gamma\in \mathcal C(p+1,\alpha)$ for some $p\geq 0$ and $\alpha\in(0,1)$ such that $\beta=2(p+\alpha)>1$
       (see \cite[Lem.~1.5]{Suetin1974}).
   \end{center}\end{minipage}}
\end{center}
\caption{Regularity of the boundary $\Gamma$, regularity of the Schwarzian derivative $S_\psi$ defined in \eqref{Sch}, and properties of the Grunsky operator.}\label{FigGrunsky}
\end{figure}

In Figure~\ref{FigGrunsky} we summarize links between smoothness of $\Gamma$, boundedness of $S_\psi$ and spectral properties of $C$, with pointers to the literature.

It is not difficult to see that the Grunsky matrix does not depend on a linear transformation of the plane. We may therefore assume, without loss of generality, that $0\in G$, and that the inner conformal map $\varphi:\mathbb D\mapsto G$ satisfies $\varphi(0)=0$ and $\varphi'(0)=1$ (in fact the outer conformal map $\widetilde \psi:\mathbb D^*\mapsto 1/G$ is given by $\widetilde \psi(w)=1/\varphi(1/w)$ with $\mbox{cap}(1/\Gamma)=1$). We  note that, if $\Gamma$ is quasiconformal, then so is $1/\Gamma$ by \cite[Lemma~9.8]{Pommerenke1975}.

Several authors (see, e.g., \cite[\S 2.2.2, pp.~70--73]{Tao})) consider an extension of $C=C(\Gamma)$ as a linear operator acting on $\ell^2(\mathbb Z)$, namely,
$$
   A = \left[\begin{array}{cc} \widetilde C & B \\
                            B^T & C \end{array}\right],
$$
which is complex symmetric. Here $\widetilde C$ represents (up to permutation of rows and columns)  the Grunsky operator $C(1/\Gamma)$, and a generating function of $B$ in terms of the inner and outer conformal maps $\varphi$ and $\psi$ is given in \cite[\S 2.2.2, p.~73]{Tao}. A generalized Grunsky inequality shows that $A$ is unitary, provided that $G$ has positive planar Lebesgue measure and $\Gamma$ has zero planar Lebesgue measure, which is true in our setting of a quasiconformal curve $\Gamma$. In particular, it is easily seen that $M=I-C^*C= B^* B$. Jones in \cite[Theorem~1.3]{Jones1999} gives another criterion for $C(1/\Gamma)$ to be in Schatten class of index $p$ in terms of the inner conformal map $\varphi:\mathbb D\to G$. For $p=2$,
this criterion is equivalent to requiring that $\int_{\mathbb D} |\log \varphi'(w)|^2 dA(w) <\infty$.

\section{Proofs of the main theorems}

\begin{proof}[Proof of Theorem~\ref{thm1}]
    Let $K$ be a compact subset of $\Omega$, and thus $r:=\min \{ |\phi(z)|: z \in K\} > 1$.
    Using \eqref{faber_residual} and \eqref{basis2} we have for $z=\psi(w)\in K$ that
  \begin{align*}
    \sqrt{\frac{\pi}{n+1}} \frac{p_n(z)-\fa_n(z)}{\phi'(z)\phi(z)^n}
   & =
    \sqrt{\frac{\pi}{n+1}} \frac{\psi'(w)}{w^n} \Bigl( \fa_0(z),...,\fa_n(z)\Bigr) \Bigl( R_{n+1}^{-1} - I_{n+1} \Bigr) e_n
    \\&
    = \sqrt{\frac{\pi}{n+1}} \frac{1}{w^n} \Bigl( \fe_0(w),...,\fe_\ell(w),...\Bigr)
    \Bigl( R^{-1} - I \Bigr) e_n \\
    &\quad+ \sqrt{\frac{\pi}{n+1}}\sum_{j=0}^n \sqrt{\frac{j+1}{\pi}}\frac{w^j}{w^n} \Bigl(R_{j,n}^{-1} - \delta_{j,n}\Bigr).
  \end{align*}

    The right-hand side will be divided up into three terms, two of them being bounded above uniformly for $z\in K$ by $\mathcal O(\sqrt{\varepsilon_n}/r^{n/2})$, and the last one by $\mathcal O({\sqrt{\varepsilon_n}}/{n^{\beta/2}})$, as this will be sufficient for the claim of Theorem~\ref{thm1}. Observe that, by \eqref{faber_residual} and the use of the Cauchy-Schwarz inequality we have
    \begin{align*}
        \Bigl| \sqrt{\frac{\pi}{n+1}}& \frac{1}{w^n} \Bigl( \fe_0(w),...,\fe_\ell(w),...\Bigr)
       \Bigl( R^{-1} - I \Bigr) e_n \Bigr|
       \\&=
        \sqrt{\frac{\pi}{n+1}}\frac{1}{r^n}  \, \Bigl| \Bigl( \sqrt{\frac{1}{\pi}} w^{-2} ,...,\sqrt{\frac{\ell+1}{\pi}}w^{-\ell-2},...\Bigr)
        C R^{-1}
       \Bigl( I - R \Bigr) e_n \Bigr|
       \\&\leq
        \sqrt{\frac{\pi}{n+1}}\frac{1}{r^n}  \, \Bigl\| \Bigl( \sqrt{\frac{1}{\pi}} w^{-2} ,...,\sqrt{\frac{\ell+1}{\pi}}w^{-\ell-2},...\Bigr)
        \Bigr\| \, \|C \| \,\|R^{-1} \| \,
       \Bigl\| (I-R)e_n \Bigr\|
       \\&\leq
        \sqrt{\frac{1}{n+1}}\frac{1}{r^n}  \,  \frac{\|C \| \,
        \|R^{-1} \|}{|w|^2-1} \sqrt{\frac{\varepsilon_n \, \| C \|^2 }{1-\| C \|^2}} = \mathcal O(\frac{\sqrt{\varepsilon_n/n}}{r^{n}}),
    \end{align*}
    where in the last inequality we have applied \eqref{bound_R_row} and \eqref{y(w)}.

    Let $m$ be the integer part of $(n+1)/2$. Then, again by the Cauchy-Schwarz inequality and \eqref{bound_R_row}, we have
    \begin{align*}
        \Bigl| \sqrt{\frac{\pi}{n+1}}\sum_{j=0}^{m-1} \sqrt{\frac{j+1}{\pi}}\frac{w^j}{w^n} \Bigl(R_{j,n}^{-1} - \delta_{j,n}\Bigr) \Bigr|
        &= \Bigl| w^{-n} ( \sqrt{\frac{1}{n+1}}w^0,..., \sqrt{\frac{m+1}{n+1}}w^m,0,0,...) (R^{-1} - I)e_n \Bigr|        \\ &\leq  r^{-n} \, \| ( r^0,..., r^m,0,0,...)\| \, \| R^{-1} \|
        \, \| (I-R)e_n \|
        \\&\leq  \frac{r^{m-n}}{\sqrt{r^2-1}} \, \| R^{-1} \|
        \, \sqrt{\frac{\varepsilon_n \, \| C \|^2 }{1-\| C \|^2}}
        = \mathcal O(\frac{\sqrt{\varepsilon_n}}{r^{n/2}}).
    \end{align*}

    Now, using \eqref{bound_R_entry} we obtain for the remaining term that
    \begin{eqnarray*}
        \Bigl| \sqrt{\frac{\pi}{n+1}}\sum_{j=m}^{n} \sqrt{\frac{j+1}{\pi}}\frac{w^j}{w^n} \Bigl(R_{j,n}^{-1} - \delta_{j,n}\Bigr) \Bigr|
        \leq  \sqrt{\frac{\varepsilon_n}{n+1}} \frac{1}{1-\| C \|^2}\sum_{j=m}^{n} \sqrt{(j+1)\varepsilon_j} \, r^{j-n}.
    \end{eqnarray*}
    By our assumption on $\varepsilon_n$ there exists some constant $c$ such that
    $\varepsilon_n\leq c/(n+1)^\beta$ and thus, uniformly for $m\leq j \leq n$,
    $$
           \varepsilon_j \leq \frac{c}{(j+1)^\beta} \leq \frac{c}{(m+1)^\beta} \leq \frac{2^\beta c}{(n+1)^\beta}.
    $$
    Hence,
    \begin{align*}
       \sqrt{\frac{\varepsilon_n}{n+1}} \frac{1}{1-\| C \|^2}\sum_{j=m}^{n} \sqrt{(j+1)\varepsilon_j|} \, r^{j-n}
        &\leq \sqrt{\frac{\varepsilon_n}{(n+1)^\beta}} \, \frac{\sqrt{2^\beta c}}{1-\| C \|^2} \sum_{j=m}^{n} r^{j-n}
       \\&= \mathcal O(\frac{\sqrt{\varepsilon_n}}{n^{\beta/2}}),
    \end{align*}
    as claimed above.

    Finally, by \eqref{Faber_outside_bis} and \eqref{Faber_outside} we obtain
    $$
        \sqrt{\frac{\pi}{n+1}} \frac{\fa_n(z)}{\phi'(z)\phi(z)^n}
        - 1 = \mathcal O(\frac{\sqrt{\varepsilon_n/n}}{r^n})
    $$
    uniformly for $z\in K$. Since $r^{-n}=\mathcal O(n^{-\beta})$ for all $\beta>0$, the required relation \eqref{eqn.thm1} follows.
\end{proof}

\def\fH{G}
\begin{proof}[Proof of Theorem~\ref{thm2}]
   We start by deriving a well-known formula for multiplication with $z$ for our reference polynomials. Since $\fa_n$ is of degree $n$, there exist $\fH_{j,n}\in \mathbb C$, $\fH_{j,n}=0$, for $j>n+1$, such that
   \begin{equation} \label{zFaber}
      z \fa_n(z) = \sum_{j=0}^{n+1} \fa_j(z) \fH_{j,n} .
   \end{equation}

   It turns out that explicit formulas can be given for the entries $\fH_{j,n}$ in terms of the conformal map $\psi$. To see this, we use the generating function
   \cite[Section 3.1, Eqn.(4)]{Pommerenke1975}
   \begin{equation} \label{generating_Faber}
        \frac{1}{\psi(u)-z}
        = \sum_{n=0}^\infty \frac{F'_{n+1}(z)}{(n+1) u^{n+1}}
        = \sum_{n=0}^\infty \sqrt{\frac{\pi}{n+1}}\frac{\fa_{n}(z)}{u^{n+1}}.
   \end{equation}
   Then, by multiplying by $\psi(u)-z$ and comparing like powers of $u$ we conclude that, for $j \leq n+1$,
   $$
        \fH_{j,n} = \sqrt{\frac{n+1}{j+1}} \psi_{j-n} .
   $$

    For relating $H$ with $\fH$, we denote by $\underline{H}_n$, $\underline{I}_n$, and $\underline{\fH}_n$ the submatrices of $H$, $I$, and $\fH$, respectively, formed with the first $n+1$ rows and the first $n$ columns.
    Then, we rewrite \eqref{multiplication} as
    $$
         \Bigl( p_0(z),...,p_{n}(z) \Bigr) (z \underline{I}_n - \underline{H}_n) = (0,...,0),
    $$
    This, in view of \eqref{basis} and \eqref{basis2}, leads to the relations
    \begin{align*}
        (0,...,0) &= \Bigl( p_0(z),...,p_{n}(z) \Bigr) (z \underline{I}_n - \underline{H}_n) R_{n}
        \\&=
        \Bigl( \fa_0(z),...,\fa_{n}(z) \Bigr) R_{n+1}^{-1} (z \underline{I}_n - \underline{H}_n) R_{n}
        \\&=
        \Bigl( \fa_0(z),...,\fa_{n}(z) \Bigr) (z \underline{I}_n - R_{n+1}^{-1} \underline{H}_n R_{n} ) .
    \end{align*}

    By the uniqueness of the coefficients in \eqref{zFaber} we conclude that
    $R_{n+1}\underline{\fH}_n = \underline{H}_n R_{n}$ for all $n$.
    Since below $\underline{\fH}_n$ in $\fH$ and below $R_n$ in $R$
    all entries are equal to zero, we conclude that
    \begin{equation} \label{link_shift}
            R \fH = H R ,
    \end{equation}
    which together with Corollary~\ref{cor_Grunsky} and the boundedness of $H$ (see the discussion following Eqn. \eqref{multiplication}) implies that $\fH$ represents a bounded operator.
    Furthermore, for all $-1 \leq k \leq n$ and $n \geq 0$,
    \begin{align*}
        H_{n-k,n} - \fH_{n-k,n} &= e_{n-k}^* \Bigl( H - \fH \Bigr) e_n
        \\&=
        e_{n-k}^* \Bigl( H(I-R) - (I-R)\fH \Bigr) e_n
        \\&=
        \sum_{\ell=n-k-1}^n  H_{n-k,\ell} (I-R)_{\ell,n}
        - \sum_{\ell=n-k}^{n+1}  (I-R)_{n-k,\ell} \fH_{\ell,n}.
    \end{align*}

    Using \eqref{bound_R_entry}, we get the upper bound
    \begin{eqnarray*} &&
       | H_{n-k,n} - \fH_{n-k,n} | \leq \frac{1}{1-\| C \|^2} \Bigl(
       \| H \| \, \sum_{\ell=n-k-1}^n  \sqrt{\varepsilon_\ell\varepsilon_{n}}
        + \| \fH\| \, \sum_{\ell=n-k}^{n+1} \sqrt{\varepsilon_\ell\varepsilon_{n-k} } \Bigr),
    \end{eqnarray*}
    and the claim in Theorem~\ref{thm2} follows.
\end{proof}

\section{Numerical results}
The purpose of this section is to present numerical results that provide experimental support
to the conjecture that the rate $\mathcal{O}(1/n)$ in Theorem~\ref{thm1} is best possible for domains with piecewise analytic boundaries having corners.

To do this we test the conjecture numerically by constructing a finite sequence of Bergman
polynomials associated with a very simple geometry. More precisely, we choose $G$ to be defined
by the two intersecting circles $|z-1|=\sqrt{2}$ and $|z+1|=\sqrt{2}$,
which meet orthogonally at the points $i$ and $-i$,

It is trivial to check that, in this case, the associated conformal map
$\phi: \Omega\to\mathbb{D}^*$ is given by
\begin{equation}\label{eq:Phi-CC}
\phi(z)=\frac{1}{2}\left(z-\frac{1}{z} \right).
\end{equation}

Let $A_n(z)$ denote the error in the approximation \eqref{eqn.thm1}, i.e., let
\begin{equation*}
A_n(z):=\sqrt{\frac{\pi}{n+1}} \frac{p_n(z)}{\phi'(z)\phi(z)^n} - 1
\end{equation*}
Then, in Tables~\ref{tab1} and \ref{tab2} we report the computed values of $A_n(3)$ and $A_n(2i)$,
which we believe to be correct to all figures quoted, for $n=100,\dots,120$, in two columns of even and odd values of $n$.
We also report the values of the parameter $s$, which is designed to test the hypothesis
\begin{equation} \label{hypothesis}
|A_n(z)|\asymp \frac{c}{n^s},\quad z\in\Omega,
\end{equation}
for some positive constants $c$.
This was done by estimating $s$ by means of the formula
$$
s_n:=\log\left(\frac{|A_n(z)|}{|A_{n+2}(z)|}\right)\big/\log\left(\frac{n+2}{n}\right).
$$

All the computations presented here were carried out in Maple 16 with 128 significant figures on a MacBook Pro.
The construction of the Bergman polynomials was made by using the Arnoldi Gram-Schmidt algorithm;
see \cite[Section 7.4]{Nikos2013} for a discussion regarding the
stability of the algorithm.

\begin{table}[t]
\begin{center}
\begin{tabular}{ccc|ccc}
\hline $ $
$n$ & $|A_n(3)|$  & $s_n$ &   $n$  & $|A_n(3)|$ &$s_n$  ${\vphantom{\sum^{\sum^{\sum^N}}}}$ \\*[3pt] \hline
100 & 8.120e-5 & 1.02301& 101  & 7.210e-5  & 0.9537 \\
102 & 7.958e-5 & 1.02284& 103  & 7.077e-5  & 0.9543\\
104 & 7.801e-5 & 1.02266&105  &  6.948e-5  & 0.9549\\
106 & 7.651e-5 & 1.02249&107  &  6.824e-5  & 0.9555\\
108 & 7.506e-5 & 1.02233&109  &  6.704e-5  & 0.9561\\
110 & 7.366e-5 & 1.02216& 111  & 6.589e-5  & 0.9567 \\
112 & 7.232e-5 & 1.02200& 113  & 6.477e-5  & 0.9572 \\
114 & 7.102e-5 & 1.02184& 115  & 6.369e-5  & 0.9577 \\
116 & 6.977e-5 & 1.02169& 117  & 6.265e-5  & 0.9582\\
118 & 6.856e-5 & 1.02154& 119  & 6.164e-5  & --- \\
120 & 6.739e-5 & --- &   &  &  \\ \hline
\end{tabular}
\end{center}
\medskip
\caption{Computed values for $|A_n(3)|$ and $s_n$.}
\label{tab1}
\end{table}

\begin{table}[h]
\begin{center}
\begin{tabular}{ccc|ccc}
\hline $ $
$n$ & $|A_n(2i)|$  & $s_n$ &   $n$  & $|A_n(2i)|$ &$s_n$  ${\vphantom{\sum^{\sum^{\sum^N}}}}$ \\*[3pt] \hline
100 & 1.323e-3 & 1.02551& 101  & 1.171e-3  & 0.9513 \\
102 & 1.297e-3 & 1.02526& 103  & 1.150e-3  & 0.9521\\
104 & 1.271e-3 & 1.02503&105   & 1.129e-3  & 0.9528\\
106 & 1.246e-3 & 1.02480&107   & 1.109e-3  & 0.9534\\
108 & 1.223e-3 & 1.02457&109   & 1.089e-3  & 0.9541\\
110 & 1.200e-3 & 1.02435& 111  & 1.071e-3  & 0.9547 \\
112 & 1.178e-3 & 1.02413& 113  & 1.052e-3  & 0.9553 \\
114 & 1.157e-3 & 1.02392& 115  & 1.035e-3  & 0.9559 \\
116 & 1.136e-3 & 1.02372& 117  & 1.018e-3  & 0.9565 \\
118 & 1.117e-3 & 1.02351& 119  & 1.002e-3  & --- \\
120 & 1.098e-3 & --- &   &  &  \\ \hline
\end{tabular}
\end{center}
\medskip
\caption{Computed values for $|A_n(2i)|$ and $s_n$.}
\label{tab2}
\end{table}

It is interesting to note the following regarding the presented results:
\begin{itemize}
\item
For even $n$ the values of $s_n$  decay monotonically to $1$.
\item
For odd $n$ the values of $s_n$  increase monotonically $1$.
\item
For both even and odd $n$ the values of $|A_n(3)|$ and $|A_n(2i)|$ decay monotonically to $0$.
\end{itemize}

The values of the parameter $s_n$ in the two tables, and also in some other experiments
not presented here, indicate clearly that
\eqref{hypothesis} holds with $s=1$.

\bigskip

\noindent
{\bf Acknowledgements.} Parts of this work have been carried out during a visit of B.B.\ at the University of Cyprus in October 2014 and of N.S.\ at the University of Lille in July 2015. Both authors are grateful to the hosting institutions for their support.

\end{document}